\documentclass[11pt,reqno]{amsart}
\usepackage{amssymb}
\usepackage{amsfonts}
\usepackage{amsmath}
\usepackage{stmaryrd}
\usepackage{physics}
\usepackage{braket}
\usepackage{dsfont}
\usepackage{bbold}
\usepackage{graphicx}
\usepackage{relsize}
\usepackage[table,xcdraw]{xcolor}
\usepackage{enumerate}
\usepackage[pagebackref, colorlinks = true, linkcolor = blue, urlcolor  = blue, citecolor = red]{hyperref}
\usepackage[margin=1in]{geometry}
\usepackage{enumitem}
\usepackage{mathtools}
\usepackage{graphbox}
\usepackage{comment}
\usepackage{float}
\usepackage[font=scriptsize]{caption}

\renewcommand{\epsilon}{\varepsilon}
\renewcommand{\phi}{\varphi}
\newcommand{\overbar}[1]{\mkern 1.5mu\overline{\mkern-1.5mu#1\mkern-1.5mu}\mkern 1.5mu}

\newtheorem{theorem}{Theorem}[section]
\newtheorem{definition}[theorem]{Definition}
\newtheorem*{definition*}{Definition}

\newtheorem{corollary}[theorem]{Corollary}
\newtheorem{lemma}[theorem]{Lemma}
\newtheorem{remark}[theorem]{Remark}

\newtheorem*{conjecture*}{Conjecture}

\theoremstyle{definition}
\newtheorem{example}[theorem]{Example}

\definecolor{darkgreen}{rgb}{0,0.392,0}

\author{Satvik Singh}
\email{satviksingh2@gmail.com}
\address{\parbox{\linewidth}{Department of Physical Sciences,\\
Indian Institute of Science Education and Research (IISER) Mohali, Punjab, India.}}
\title[How to define your dimension?]{How to define your dimension \\ \normalfont{A \lowercase{discourse on} H\lowercase{ausdorff dimension and self-similarity}}}

\begin{document}

\begin{abstract}
    One often distinguishes between a line and a plane by saying that the former is one-dimensional while the latter is two. But, what does it mean for an object to have $d-$dimensions? Can we define a consistent notion of dimension rigorously for arbitrary objects, say a snowflake, perhaps? And must the dimension always be integer-valued? After highlighting some crucial problems that one encounters while defining a sensible notion of dimension for a certain class of objects, we attempt to answer the above questions by exploring the concept of Hausdorff dimension -- a remarkable method of assigning dimension to subsets of arbitrary metric spaces. In order to properly formulate the definition and properties of the Hausdorff dimension, we review the critical measure-theoretic terminology beforehand. Finally, we discuss the notion of self-similarity and show how it often defies our quotidian intuition that dimension must always be integer-valued.
\end{abstract}

\maketitle


\section{Introduction}
Intuitively, the dimension of a set can be thought of as the minimum number of parameters needed to specify a unique element within the set. Following this reasoning, one can roughly think of points, curves, surfaces and solids to have, respectively, zero, one, two and three dimensions. However, shortcomings in this simple-minded logic surface when one considers, for instance, the fact that the unit line has the same cardinality as the unit plane \cite{Cantor1877, nicolay2014building}, i.e. there exists a one-one mapping from $[0,1]$ onto $[0,1]^2$. Thus, any point in the plane can be specified by only one parameter by uniquely mapping it to its pre-image in $[0,1]$. Does this mean that the above intuition is flawed? What if we consider the additional topological (or metric) structure on these spaces? In this scenario, the seminal works of Peano \cite{Peano1890}, Hilbert \cite{Hilbert1891space} and others \cite{bader2012space} show that the simplest \emph{space-filling curves} (continuous mappings from the unit line \emph{onto} the unit plane, see Figure~\ref{fig:Hilbert}) can be used to ``fill up" two dimensions from a supposed one-dimensional set, thus apparently undermining the above notion of dimension again. 

\begin{figure}[H]
    \centering
    \includegraphics[scale=0.46,align=c]{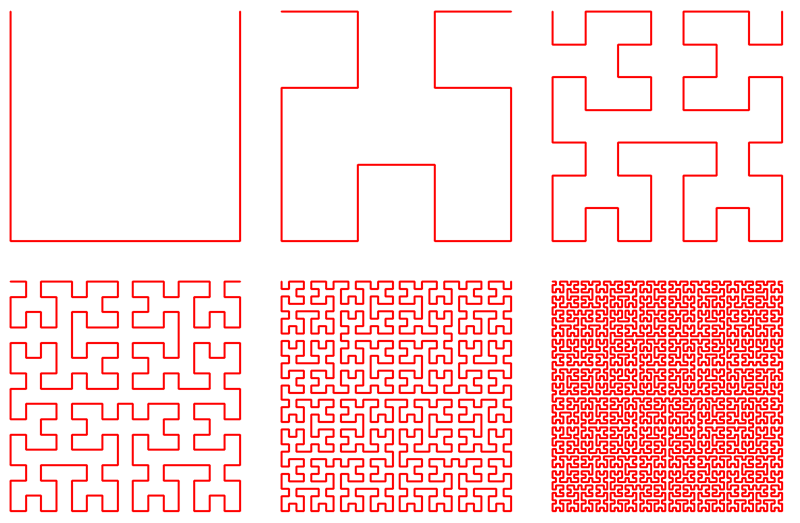}
    \caption{Ranges of the first six iterations of ``Psuedo-Hilbert" curves, whose pointwise limit yields the continuous space-filling Hilbert curve from $[0,1]$ onto $[0,1]\times [0,1]$}
    \label{fig:Hilbert}
\end{figure}

However, such space-filling curves are bound to be non-injective, meaning that although each point in the unit square is hit atleast once by the curve, its pre-image on the unit line is not unique. More generally, it can be shown that $[0,1]$ and $[0,1]^2$ are not topologically \emph{homeomorphic}, i.e. there is no continuous bijection between the two which can also be continuously inverted. The concept of \emph{topological dimension} \cite[Chapter 3]{edgar2008measure} uses this fact to dimensionally distinguish between the two spaces. Roughly speaking, the topological dimension of a topological space is the minimum $d$ such that, given any open cover of the space, there exists a refinement with the property that no point in the space lies in the intersection of more than $d+1$ covering sets in the refinement. In this language, the unit line has topological dimension $d=1$, since any open cover of it is bound to count some points twice. Similarly, it is not too hard to see that the unit plane has topological dimension $d=2$. Clearly, \emph{homeomorphic} spaces have identical topological dimensions. However, is invariance under homeomorphisms really a desirable property that every notion of dimension must posses? 

To answer the above question, let us consider the graph $\Gamma(f)\coloneqq \{ (x,f(x))\in \mathbb{R}^3 : x\in [0,1] \}$ of some space-filling curve $f:[0,1]\rightarrow [0,1]^2$. Intuitively, one would expect that the dimension of $\Gamma(f)$ must atleast be greater than one, because the range of $f$ completely fills the unit plane (it is not hard to imagine that $\Gamma (f)$ is a ragged object in the unit cube, which looks like the unit square when seen head on from a direction parallel to the $x$-axis)! However, since the graph of any continuous function is homeomorphic to its domain, the topological dimension of $\Gamma(f)$ is again one. This is just one example of a large family of \emph{self-similar} objects (or \emph{fractals}), for which the topological dimension is smaller than what one might expect. The problem with the topological notion of dimension is that it is completely oblivious to any \emph{local structure} (or roughness) that the concerned object might posses, since homeomorphisms can only discriminate between those spaces which \emph{cannot} be transformed into one another by means of a \emph{global} continuous deformation. Let us consider the self-similar \emph{Koch snowflake} from Figure~\ref{fig:koch} to further elaborate on this point. While it is evident that the curve doesn't fill up the entire plane, it is definitely more complex than a line because of all the intricate spikes and crevices. Being impervious to this additional structure, a topologist would again call this a one-dimensional curve, since it can be shown to be homeomorphic to an interval in $\mathbb{R}$. It is thus clear that a more sophisticated notion of dimension needs to be developed in order to deal with such exotic objects. 

\begin{figure}[H]
    \centering
    \includegraphics[scale=0.43,align=c]{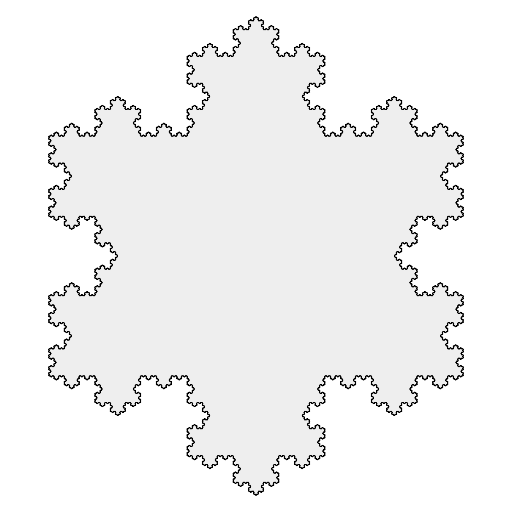}
    \caption{The $7^{\text{th}}$ level in the construction of the Koch snowflake.}
    \label{fig:koch}
\end{figure}

\begin{remark}
One can also exploit the linear structure of the real line $\mathbb{R}$ and the real plane $\mathbb{R}^2$ to dimensionally distinguish between them. In this regime, the notion of \emph{vector space dimension} becomes important, which, for any \emph{finite-dimensional vector space}, is just the (unique) cardinality of its basis set. The structure preserving maps in the category of vector spaces are \emph{linear bijections}, and consequently, the vector space dimension stays \emph{invariant} under them. One can extend this idea of linear dimension to spaces which only locally look like $\mathbb{R}^d$ to obtain the class of $d$-dimensional \emph{manifolds}. However, this notion of dimension is too restrictive in its domain, since many objects occuring in nature are not manifolds! Or, as Mandelbrot would say \cite[p. 1]{mandelbrot1983fractal}, “clouds are not spheres, mountains  are  not  cones,  coastlines  are  not  circles,  and  bark  is  not smooth, nor does lightning travel in a straight line.” In this regard, it would be instructive for the reader to show that the closed unit interval $[0,1]$, for instance, is neither a vector space nor a manifold.
\end{remark}

During the early $20^{\text{th}}$ century, a remarkable method was developed by Hausdorff (along with Lebesgue) for assigning dimension to subsets of an arbitrary metric space, which not only generalizes the topological notion but also takes the sets' local structure into account. The idea is to explicitly employ the underlying metric (or distance function, which is sensitive to the object's local structure at small scales) in order to associate a ``proper" kind of measurement to the concerned object, which will be used to measure its size. For instance, if we consider an interval $(a,b)\subset \mathbb{R}$, we can say that its cardinality ($\operatorname{dim}=0$) is infinite, its length ($\operatorname{dim}=1$) is $b-a$ and its area ($\operatorname{dim}=2$) is zero. The step $d$ at which the size of an object changes from infinity to zero serves as an excellent candidate for its dimension. For the interval $(a,b)$, this step is $d=1$. Moreover, if the step size for deciding this proper measurement is allowed to be arbitrarily small, then nothing stops the dimension of any given object from taking an arbitrary real value between zero and infinity. The only challenge is to come up with a meaningful notion of the $\delta$-dimensional size of a given object for every $\delta>0$.

In this article, we will formalize the above idea of \emph{Hausdorff dimension} for subsets of arbitrary metric spaces. In order to do so, it is essential to first gather the appropriate measure-theoretic machinery, which will form the basis of our discussion in Section~\ref{sec:measure}. Section~\ref{sec:hausdorff} forms the core of this article, where we develop the promised definition of the Hausdorff dimension and elucidate several of its properties. Finally, we describe the notion of self-similarity for subsets of Euclidean spaces $\mathbb{R}^k$ in Section~\ref{sec:eg} and compute the Hausdorff dimension of several famous self-similar objects like the \emph{Cantor set}, the \emph{Koch curve} and the \emph{Sierpiński triangle}. It is probably worthwhile to emphasis here that this article is only meant to provide a light and accessible introduction to the basics of dimension theory and self-similarity. Hence, we do not claim to be exhaustive in our account of the aforementioned topics, neither do we claim to report on any modern advancements in the field. For a more sophisticated discussion on these topics, we refer the readers to the wonderful articles \cite{GRASSBERGER1985Haus, Manin2005TheNO,Dierk2005Haus,Elekes2011Haus,fernandex2016Haus} and books \cite{mandelbrot1983fractal, falconer1990fractal, peitgen1992chaos, edgar2008measure}.


\section{Measure theoretic preliminaries} \label{sec:measure}
In this section, we'll briefly review some key notions from measure theory, which will be instrumental in defining the notion of Hausdorff dimension in the next section. For a much more thorough introduction to measure theory, the readers are referred to the classic texts \cite{rudin1987real,billingsley1986probability,halmos1976measure}. Before proceeding with our discussion, let us quickly review some basic notation. Given a non-empty set $X$, we denote its power set and the empty subset by $\mathcal{P}(X)$ and $\emptyset$, respectively. The complement of a subset $A\subseteq X$ is denoted by $X\setminus A$. The non-negative reals along with infinity are denoted by the interval $[0,\infty]$.

In a nutshell, measure theory deals with the task of sensibly assigning volumes to the subsets of a given set $X$. For $X=\mathbb{R}$, for instance, one would like to get hold of a function (called a \emph{measure}) $\mu: \mathcal{P}(\mathbb{R})\rightarrow [0,\infty]$ such that it is an extension of the usual notion of length for intervals, i.e. $\sigma ((a-b]) = b-a \,\,\, \forall a,b\in \mathbb{R}$. It is also rational  to demand that this function be \emph{countably additive} and \emph{translationally invariant}, meaning that shifting a subset by a real number should not change its volume and that the volume of countable disjoint unions of subsets should just be the sum of individual volumes of the subsets. However, by using the \emph{axiom of choice}, one can show that such a function cannot exist, see \cite[Section 16]{halmos1976measure}! It turns out that asking for a measure to be defined on all subsets of the real line is too strong a demand. Hence, one needs to restrict the domain of definition, which is where the notion of $\sigma$-algebras come into picture. 

\begin{definition}
A collection $\mathcal{F}$ of subsets of a set $X$ is called a $\sigma$-\emph{algebra} if:
\begin{itemize}
    \item $\emptyset, X \in \mathcal{F}$.
    \item $A\in \mathcal{F}\implies X\setminus A\in \mathcal{F}$.
    \item $\{A_n\}_{n\in \mathbb{N}}\subseteq \mathcal{F} \implies \bigcup_{n\in \mathbb{N}}A_n \in \mathcal{F}$.
\end{itemize}
\end{definition}

Using the last two properties, one can easily deduce that $\sigma$-algebras are also closed under countable intersections. The set $X$ along with a $\sigma$-algebra $\mathcal{F}$ is said to define a \emph{measurable} space $(X,\mathcal{F})$. The subsets in $\mathcal{F}$ are then called \emph{measurable}. Given a collection $\mathcal{C}$ of subsets of a set $X$, it is natural to ask for a minimal $\sigma$-algebra on $X$ which contains $\mathcal{C}$. For instance, suppose we have a topology $\tau$ (a collection of open sets) defined on $X$ and we want to construct a $\sigma$-algebra on $X$ such that it contains $\tau$. The most economical way to do this is to obviously consider the minimal such $\sigma$-algebra. The following result assures us that such a minimal $\sigma$-algebra exists for every $\mathcal{C}$.

\begin{lemma}
Let $\mathcal{C}$ be a non-empty collection of subsets of a set $X$. Define $$\sigma (\mathcal{C}) \coloneqq \bigcap \{\mathcal{F} : \mathcal{F} \text{ is a } \sigma\text{-algebra which contains } \mathcal{C} \}. $$ Then, $\sigma (\mathcal{C})$ is the minimal $\sigma$-algebra which contains $\mathcal{C}$, i.e. if $\mathcal{G}$ is a $\sigma$-algebra which contains $\mathcal{C}$, then $\sigma (\mathcal{C}) \subseteq \mathcal{G}$. We say that the collection $\mathcal{C}$ generates $\sigma(\mathcal{C})$.
\end{lemma}

Given an arbitrary topological space $(X,\tau)$, we can turn it into a measurable space by defining $\mathcal{F}=\sigma(\tau)$ to be the $\sigma$-algebra generated by the topology $\tau$. The measurable sets are then called \emph{Borel} sets. It is easy to see that all open, closed and compact sets (among others) in $X$ are Borel. 

With the concept of $\sigma$-algebra in place, we now arrive at the central definition of a measure. 
\begin{definition}
Consider a measurable space $(X,\mathcal{F})$. Then, a set function $\mu : \mathcal{F}\mapsto [0,\infty]$ is said to be a \emph{measure} if $\mu (\emptyset) = 0$ and if $\mu$ is countably additive, i.e. for all countable pairwise disjoint collections of measurable subsets $\{A_n\}_{n\in \mathbb{N}}\subseteq \mathcal{F}$, $\mu (\bigcup_{n\in \mathbb{N}} A_n ) = \sum_{n=1}^\infty \mu (A_n)$.
\end{definition}
Notice that measurable sets are allowed to have infinite measure. A measurable space $(X,\mathcal{F})$ equipped with a measure $\mu$ is called a \emph{measure space} $(X,\mathcal{F},\mu)$. We list some important properties of measures below, which are straightforward consequences of the above definition.

\begin{lemma}
Let $(X,\mathcal{F},\mu)$ be a measure space. Then,
\begin{itemize}
    \item $\mu (\cup_{i=1}^n A_n) = \sum_{i=1}^n \mu(A_n)$, where $\{A_i \}_{i=1}^n\subseteq \mathcal{F}$ is a pairwise disjoint collection of subsets.
    \item $A,B\in \mathcal{F}$ with $A\subseteq B \implies \mu (A) \leq \mu (B)$.
    \item $A,B\in \mathcal{F}$ with $A\subseteq B$ and $\mu(A) < \infty \implies \mu(B\setminus A)=\mu (B)-\mu (A)$.
\end{itemize}
\end{lemma}

More often than not, the task of transforming sets into non-trivial measure spaces is quite complicated. One often tries to bypass the difficulty of choosing an appropriate $\sigma$-algebra by constructing an \emph{outer measure} first, which is a special set function defined on \emph{all} subsets of a given set. The price to pay is to relax the countable additivity condition to just countable \emph{sub}-additivity.

\begin{definition}
An \emph{outer measure} on a set $X$ is a set function $\overbar\mu : \mathcal{P}(X)\mapsto [0,\infty]$ such that
\begin{itemize}
    \item $\overbar\mu (\emptyset) = 0$.
    \item (Monotonicity) $A\subseteq B \implies \overbar\mu (A)\leq \overbar\mu (B)$.
    \item (Countable sub-additivity) $\{A_n\}_{n\in \mathbb{N}}\subseteq \mathcal{P}(X)\implies \overbar\mu (\bigcup_{n\in \mathbb{N}} A_n ) \leq \sum_{n=1}^\infty \mu (A_n)$.
\end{itemize}
\end{definition}

Recall that a collection $\mathcal{A}$ of subsets of $X$ is said to cover a subset $B\subseteq X$ if $B\subseteq \bigcup_{A\in \mathcal{A}} A$. Now, if one already has a set function $f:\mathcal{A}\mapsto [0,\infty]$ defined on a family $\mathcal{A}$ of subsets of a given set $X$ which covers $X$, then the following theorem provides a powerful method to construct an outer measure on $X$. It should be noted that the assumptions here are quite easily fulfilled in practice. The real line $\mathbb{R}$ with the family of all half-open intervals equipped with the length function $f((a,b])=b-a$ readily provides an example.

\begin{theorem} \label{theorem:outer-1}
Given an arbitrary set $X$ along with a family of covering subsets $\mathcal{A}$ and a set function $f:\mathcal{A}\rightarrow [0,\infty]$, there exists a unique outer measure $\overbar\mu$ on $X$ such that 
\begin{itemize}
    \item $\overbar\mu(A) \leq f(A) \,\, \forall A\in \mathcal{A}$.
    \item if $\overbar \nu$ is another outer measure on $X$ with $\overbar\nu (A)\leq f(A) \,\, \forall A\in \mathcal{A}$, then $\overbar\nu (B)\leq \overbar\mu (B) \,\, \forall B\subseteq X$.
\end{itemize}
Moreover, $\overbar\mu$ is given by the following formula: $\overbar\mu (B) = \operatorname{inf} \sum_{A\in \mathcal{C}} f(A)$, where the infimum is taken over all countable covers $\mathcal{C}$ of $B$ by sets of $\mathcal{A}$. 
\end{theorem}

One can think of $f$ in the above theorem as the candidate for defining the measure of certain subsets of $X$. Then, the theorem presents a way to assign an outer measure to these subsets which is as large as possible, with the constraint that it should not exceed the candidate values. Now, given an outer measure $\overbar\mu$ on a set $X$, there is a canonical way to construct a $\sigma$-algebra $\mathcal{F}$ on $X$ and an associated measure $\mu$. We present this method in the following theorem.

\begin{theorem} \label{theorem:outer-to-measure}
Let $\overbar\mu$ be an outer measure on a set $X$. A subset $A\subseteq X$ is said to be $\overbar\mu$-measurable if $\overbar\mu (E) = \overbar\mu (E\cap A) + \overbar\mu (E \setminus A) \,\, \forall E\subseteq X$. The collection $\mathcal{F}$ of all $\overbar\mu$-measurable subsets of $X$ forms a $\sigma$-algebra on $X$ and the restriction of $\overbar\mu$ to $\mathcal{F}$ is a measure.
\end{theorem}

At this point, it is fair to ask about the structure of the $\sigma$-algebra that is generated by any given outer measure upon an application of Theorem~\ref{theorem:outer-to-measure}. In particular, since metric spaces $X$ will be our primary spaces of interest in the following section, we would like to get hold of outer measures $\overbar \mu$ on $X$ which are good enough to guarantee that the corresponding family of all $\overbar \mu$-measurable sets at least contains the usual Borel sets generated by the standard metric topology on $X$. This is not the case in general, for outer measures supplied by Theorem~\ref{theorem:outer-1} above, see \cite[Section 5.4]{edgar2008measure} for a counterexample. Fortunately, Theorem~\ref{theorem:outer-1} can be slightly tweaked to remedy this flaw.

\begin{theorem}\label{theorem:outer-2}
Let $X$ be a metric space and $\mathcal{A}$ be a family of subsets of $X$ such that for every $x\in X$ and $\epsilon >0$, there exists $A\in \mathcal{A}$ with $\operatorname{diam}A\leq\epsilon$ such that $x\in A$. Let $f:\mathcal{A}\rightarrow [0,\infty]$ be a candidate function for the measure. Now, define $\mathcal{A}_\epsilon \coloneqq \{A\in \mathcal{A} : \operatorname{diam}A \leq \epsilon \}$ and let $\overbar \mu_\epsilon$ be the outer measure defined using the family $\mathcal{A}_\epsilon$ and the function $f$, as in Theorem~\ref{theorem:outer-1}. Finally, define \begin{equation*}
    \overbar \mu (B) = \lim_{\epsilon \rightarrow 0} \overbar \mu_\epsilon (B) \qquad \forall B\subseteq X
\end{equation*}
Then, $\overbar \mu$ is an outer measure on $X$. Moreover, all Borel subsets of $X$ are $\overbar \mu$-measurable.
\end{theorem}

Let us pause for a moment to appreciate the content of the above theorem. Starting from a simple initial candidate function $f$ defined on some suitable family of subsets of $X$, we now have a systematic procedure to construct a rich $\sigma$-algebra $\mathcal{F}$ containing all the Borel sets and a measure $\mu:\mathcal{F}\rightarrow [0,\infty]$ on $X$. This procedure will be at the heart of the definition of the Hausdorff measure in the next section. 

\begin{remark}[Lebesgue measure]
Before concluding this section, let us remark that the Lebesgue measure on $\mathbb{R}^k$ is constructed by applying Theorems~\ref{theorem:outer-1} and \ref{theorem:outer-to-measure} on the family of all half-open rectangles of the form 
\begin{equation*}
    R = (a_1,b_1]\times (a_2,b_2]\times \ldots \times (a_k,b_k] \subseteq \mathbb{R}^k
\end{equation*}
and the candidate function $f(R)=\prod_{i=1}^k (b_i-a_i)$, where $a_i<b_i$ for each $i$. It can be shown that all Borel sets in $\mathbb{R}^k$ are also Lebesgue measurable (notice that in this special case, the construction elaborated in Theorem~\ref{theorem:outer-2} in not required).
\end{remark}

\section{Hausdorff dimension} \label{sec:hausdorff}
Equipped with all the required measure-theoretic tools from the previous section, we can now begin to develop the notion of Hausdorff dimension for subsets of an arbitrary metric space $X$. The process unfolds in two steps. As was remarked in the introduction, the first (and the most challenging step) is to define a sound notion of $\delta$-measurements which can be used to quantify the size of an object for each $\delta>0$. The previous section informs us that measures are built to fulfil this precise function. However, as noted before, defining a measure from scratch can be an arduous task, so we start by formulating the definition of $\delta$-\emph{Hausdorff outer measures}. Note that unless stated otherwise, $X$ will always denote a metric space in this section. Recall that a metric space $X$ is a set equipped with a real valued function $d$ on $X\times X$ such that $d(x,y)\geq 0$, $d(x,y)=d(y,x)$, $d(x,y) \leq d(x,z)+d(z,y)$ for all $x,y,z\in X$ and $d(x,y)=0$ if and only if $x=y$. The diameter of a subset $A\subseteq X$ is defined as $\operatorname{diam}A \coloneqq \operatorname{sup}\{ d(x,y) : x,y\in A \}$. For Euclidean spaces $X=\mathbb{R}^k$, we use the standard metric $d(x,y)=||x-y||=\sqrt{\sum_{i=1}^k (x_i-y_i)^2 }$.

\begin{definition}
Let $X$ be a metric space and $\mathcal{A}=\mathcal{P}(X)$ be the collection of all subsets of $X$. For each $\delta>0$, let $f_\delta:\mathcal{A}\rightarrow [0,\infty]$ be the candidate function defined as $f(A)=(\operatorname{diam}A)^\delta \,\,\, \forall A\subseteq X$. Then, the $\delta$-\emph{Hausdorff outer measure} $\overbar H^\delta$ is constructed by applying Theorem~\ref{theorem:outer-2} on $X, \mathcal{A}$ and $f_\delta$.
\end{definition}

Using Theorems~\ref{theorem:outer-1} and \ref{theorem:outer-2}, we spell out the above definition in a bit more detail. For every $\epsilon, \delta >0$, we define the family of subsets $\mathcal{A}_\epsilon = \{ A\subseteq X : \operatorname{diam}A \leq \epsilon \}$ and set functions $f_\delta : \mathcal{A}_\epsilon \rightarrow [0,\infty]$ such that $f_\delta (A) = (\operatorname{diam}A)^\delta \,\,\, \forall A\in \mathcal{A}_\epsilon$. Now, the outer measure $\overbar H^\delta_\epsilon$ on $X$ is defined by applying Theorem~\ref{theorem:outer-1} on the set family $\mathcal{A}_\epsilon$ and the set function $f_\delta : \mathcal{A}_\epsilon \mapsto [0,\infty]$. More precisely, 
\begin{equation}
\forall B\subseteq X, \qquad  \overbar H^\delta_\epsilon (B) = \inf_{\mathcal{C}\subseteq \mathcal{A}_\epsilon} \sum_{A\in \mathcal{C}} (\operatorname{diam}A)^\delta
\end{equation}
where the infimum is taken over all countable $\epsilon$-covers $\mathcal{C}$ of $B$, i.e. over all countable covers $\mathcal{C}\subseteq \mathcal{A}_\epsilon$ of $B$. The $\delta$-Hausdorff outer measure on $X$ is then defined for all $B\subseteq X$ as the following limit:
\begin{equation}
    \overbar H^\delta (B) = \lim_{\epsilon\rightarrow 0} \overbar H^\delta_\epsilon (B) = \sup_{\epsilon>0} \overbar H^\delta_\epsilon (B) \in [0,\infty]
\end{equation} 

Although Theorem~\ref{theorem:outer-2} guarantees that $\overbar H^\delta$ is an outer measure for each $\delta>0$, it is instructive to look at the proof once. We begin by certifying the well-definition of $\overbar H^\delta$. If $0<\epsilon_1 \leq \epsilon_2$, it is easy to see that $A_{\epsilon_1}\subseteq A_{\epsilon_2}$. Hence, there are more permissible covers in the definition of $\overbar H^\delta_{\epsilon_2}$ when compared to $\overbar H^\delta_{\epsilon_1}$, which allows the infimum to decrease. Thus, for arbitrary $B\subseteq X$, we have $\overbar H^\delta_{\epsilon_2} (B) \leq \overbar H^\delta_{\epsilon_1}(B)$, which implies that $\overbar H^\delta_{\epsilon} (B)$ is a non-increasing function of $\epsilon\in (0,\infty)$. Now, if $\{\epsilon_n \}_{n\in \mathbb{N}}\subseteq \mathbb{R}$ converges to zero, then $\{\overbar H^\delta_{\epsilon_n}(B) \}_{n\in \mathbb{N}}$ is monotonically increasing and hence either converges to some value in $[0,\infty)$ or diverges to $\infty$.  
It is thus clear that $\overbar H^\delta$ is well defined. We now prove that $\overbar H^\delta$ is indeed an outer measure.

\begin{theorem}
For every $\delta > 0$, $\overbar H^\delta$ is an outer measure on $X$.
\end{theorem}
\begin{proof}
We begin by showing that for each fixed $\epsilon > 0$, $\overbar H^\delta_\epsilon$ is an outer measure. It is clear from the definition that $\overbar H^\delta_\epsilon (\emptyset) = 0$ and that $\overbar H^\delta_\epsilon$ is monotonous (if $A\subseteq B$, $\overbar H^\delta_\epsilon (A) \leq \overbar H^\delta_\epsilon (B)$). To show countable sub-additivity, we consider an arbitrary sequence $\{B_n \}_{n\in \mathbb{N}}$ of subsets of $X$ and let $\gamma > 0$. For each $n\in \mathbb{N}$, the definition of infimum allows us to choose a countable $\epsilon$-cover $\mathcal{C}_n$ of $B_n$ such that
\begin{equation} \label{eq:1}
\sum_{A\in \mathcal{C}_n} (\operatorname{diam}A)^\delta \leq \overbar H^\delta_\epsilon (B_n) + \frac{\gamma}{2^n}.
\end{equation} 
Now, notice that 
$$ B_n \subseteq \bigcup_{A\in \mathcal{C}_n} A \implies \bigcup_{n\in \mathbb{N}} B_n \subseteq \bigcup_{n\in \mathbb{N}} \bigcup_{A\in \mathcal{C}_n} A. $$ 
Hence, we obtain
\begin{align*}
    \overbar H^\delta_\epsilon (\bigcup_{n\in \mathbb{N}} B_n) \leq \overbar H^\delta_\epsilon (\bigcup_{n\in \mathbb{N}} \bigcup_{A\in \mathcal{C}_n} A) &\leq \sum_{n=1}^\infty \sum_{A\in \mathcal{C}_n} (\operatorname{diam}A)^\delta \\
    &\leq \sum_{n=1}^\infty (\overbar H^\delta_\epsilon (B_n) + \frac{\gamma}{2^n}) \\ 
    &= \sum_{n=1}^\infty \overbar H^\delta_\epsilon (B_n) + \gamma
\end{align*}
where, the first inequality follows from the monotonicity of $\overbar H^\delta_\epsilon$, the second follows from the definition of $\overbar H^\delta_\epsilon$ and the rest are consequences of the above inequality Eq.~\eqref{eq:1}. Since the above inequality holds for all $\gamma>0$, we infer that $\overbar H^\delta_\epsilon (\bigcup_{n\in \mathbb{N}} B_n ) \leq \sum_{n\in \mathbb{N}} \overbar H^\delta_\epsilon (B_n)$.

Now, we show that $\overbar H^\delta$ is an outer measure. As before, $\overbar H^\delta (\emptyset) = 0$ and monotonicity trivially follows from the definition. To show countable sub-addivity, we use the recently shown countable sub-additivity of $\overbar H^\delta_\epsilon$ and the fact that $\overbar H^\delta (B_n) = \sup_{\epsilon>0} \overbar H^\delta_\epsilon (B_n) \geq \overbar H^\delta_\epsilon (B_n) \,\, \forall n\in \mathbb{N}$ to deduce that
\begin{align*}
    \overbar H^\delta_\epsilon ( \bigcup_{n\in \mathbb{N}} B_n) \leq \sum_{n=1}^\infty \overbar H^\delta_\epsilon (B_n) \leq \sum_{n=1}^\infty \overbar H^\delta (B_n)
\end{align*}
Taking the limit $\epsilon \rightarrow 0$ then yields the desired inequality, thus finishing the proof.
\end{proof}

Now that we know that $\overbar H^\delta$ is an outer measure on $X$ for each $\delta > 0$, we can immediately obtain the corresponding $\delta$-\emph{Hausdorff measure} $H^\delta$ on the $\sigma$-algebra of $\overbar H^\delta$-measurable sets in $X$ by using Theorem~\ref{theorem:outer-to-measure}. Moreover, Theorem~\ref{theorem:outer-2} 
guarantees that the $\overbar H^\delta$-measurable sets are sufficiently rich in the sense that all Borel subsets in $X$ are $\overbar H^\delta$-measurable. 

\begin{remark} \label{remark:Hausdorff-Lebesgue}
The $k$-dimensional Lebesgue measure and the $k$-dimensional Hausdorff measure on $\mathbb{R}^k$ are equivalent, upto multiplication by a constant, see \cite[Theorem 6.3.6]{edgar2008measure}.
\end{remark}

In the next lemma, we describe a very useful scaling property of Hausdorff measures.

\begin{lemma} \label{lemma:hausdorff-scaling}
Let $B\subseteq \mathbb{R}^k$ and let $S:B\rightarrow \mathbb{R}^l$ be \emph{Lipschitz}, i.e. there exists a $\lambda>0$ such that $||S(x)-S(y)||\leq \lambda||x-y|| \,\,\, \forall x,y\in B$.
Then for all $ \delta >0$,
$$\overbar H^\delta (S(B)) \leq \lambda^\delta \overbar H^\delta (B),$$ where $S(B) \coloneqq \{ S(x): x\in B \}\subseteq \mathbb{R}^l$.
\end{lemma}
\begin{proof}
If $\mathcal{C}=\{A_n\}_{n\in \mathbb{N}}\subseteq \mathcal{A}_\epsilon$ is a countable $\epsilon$-cover of $B$, then $\lambda \mathcal{C} = \{\lambda A_n\}_{n\in \mathbb{N}}\subseteq \mathcal{A}_{\lambda\epsilon}$ is a countable $\lambda\epsilon$-cover for $S(B)$. Hence, we have $\overbar H^\delta_{\lambda\epsilon} (S(B)) \leq \sum_{n=1}^\infty [\operatorname{diam}(\lambda A_n)]^\delta = \lambda^\delta \sum_{n=1}^\infty (\operatorname{diam}A_n)^\delta $. Since this is true for all countable $\epsilon$-covers $\mathcal{C}\subseteq \mathcal{A}_\epsilon$ of $B$, taking infimum yields $\overbar H^\delta_{\lambda\epsilon} (S(B)) \leq \lambda^\delta \overbar H^\delta_{\epsilon} (B)$. Taking limit $\epsilon \rightarrow 0$ then proves the desired result.
\end{proof}

\begin{corollary} \label{corollary:hausdorff-scaling}
Let $S:\mathbb{R}^k\rightarrow \mathbb{R}^k$ be a \emph{similarity transformation}, i.e. $\forall x,y\in \mathbb{R}^k \,\,\,$,  $||S(x)-S(y)|| = \lambda||x - y||$ for some $\lambda >0$. Then, $\overbar H^\delta (S(B))= \lambda^\delta \overbar H^\delta (B)$ for all $\delta>0$ and $B\subseteq \mathbb{R}^k$.
\end{corollary}
\begin{proof}
Observe that since $S$ is a similarity transformation, we have $||S^{-1}(x)-S^{-1}(y)||= \lambda^{-1}||x-y||$ for all $x,y\in \mathbb{R}^k$. The desired result is then obtained by applying Lemma~\ref{lemma:hausdorff-scaling} on $S$ and $S^{-1}$.
\end{proof}

Before proceeding to the second and final step in defining the notion of Hausdorff dimension, we must study the behaviour of $\overbar H^\delta (B)$ as a function of $\delta$ for an subset $B\subseteq X$. 

\begin{lemma}
Let $X$ be a metric space, $B\subseteq X$ and $0<\delta_1 < \delta_2$. Then,
\begin{itemize}
    \item $\overbar H^{\delta_1} (B) < \infty \implies \overbar H^{\delta_2}(B)=0$. 
    \item $\overbar H^{\delta_2} (B) > 0 \implies \overbar H^{\delta_1} (B) = \infty$.
\end{itemize}

\end{lemma}
\begin{proof}
For every $\epsilon >0$ and $A\subseteq X$ with $\operatorname{diam}A\leq \epsilon$, we have 
$$ |A|^{\delta_2} = |A|^{\delta_2 - \delta_1} |A|^{\delta_1} \leq \epsilon^{\delta_2 - \delta_1} |A|^{\delta_1}.$$ 
Hence, for $B\subseteq X$, we infer that $\overbar H^{\delta_2}_\epsilon (B) \leq \epsilon^{\delta_2 - \delta_1} \overbar H^{\delta_1}_\epsilon (B)$. Now, if $H^{\delta_1} (B) < \infty$, the desired conclusion follows immediately by taking the limit $\epsilon \rightarrow 0$. The second conclusion follows similarly.
\end{proof}

The above theorem tells us that for every $B\subseteq X$, there is a critical value of $\delta$ at which $\overbar H^\delta (B)$ changes from $\infty$ to $0$, see Figure~\ref{fig:Hdim}. More precisely, there exists a \emph{unique} $d\in [0,\infty]$ such that $$\overbar H^\delta (B) = \infty \quad \forall \,\delta < d \qquad \text{and} \qquad \overbar H^\delta (B)=0 \quad \forall \, \delta > d.$$ For $\delta = d$, $\overbar H^\delta (B)$ can take any value from $0$ to $\infty$. We define this critical value as the \emph{Hausdorff dimension} of $B$.

\begin{definition}\label{def:Hdim}
Let $B\subseteq X$ be a subset of a metric space $X$. Then, the \emph{Hausdorff dimension} of $B$ is defined as 
\begin{equation*}
    \operatorname{dim}_H B = \sup \{\delta : \overbar H^\delta (B) = \infty \} = \inf \{\delta : \overbar H^\delta (B) = 0\}
\end{equation*}
\end{definition}

\begin{figure}[H]
    \centering
    \includegraphics{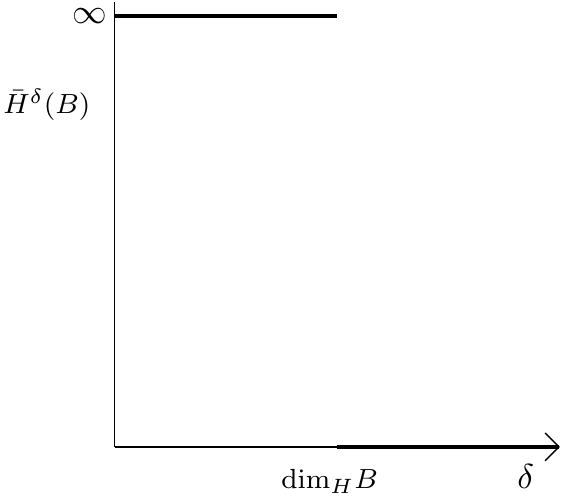}
    \caption{At $\delta=\operatorname{dim}_H B$, the value of $\overbar H^\delta (B)$ drops from infinity to zero.}
    \label{fig:Hdim}
\end{figure}

Let us now pause for a moment to reflect upon the above definition. We are now capable of assigning dimensions to subsets of arbitrary metric spaces in a way which not only generalizes the topological notion of dimension for the usual geometric shapes (see Remark~\ref{remark:haus=topo} in this regard), but at the same time also appreciates the objects' small-scale anatomy by exploiting the underlying distance function. One can think of the Hausdorff dimension as a mathematically rigorous way to encapsulate the traditional intuition that the volume of a surface, for instance, is zero or that the length of a surface is ``infinite". If someone hands you a curve in a plane, it is absurd to talk about its (two-dimensional) area. On the other hand, it is perfectly meaningful to talk about the curve's (one-dimensional) length. The $\delta$-Hausdorff measures capture this idea by allowing one to measure the $\delta$-dimensional size of an object for all $\delta>0$ and not just for integers. In this formulation, the Hausdorff dimension of an object is the right value of $\delta$ such that it becomes meaningful to talk about the $\delta$-dimensional size of the object. Since there is no restriction on $\delta>0$ to be an integer, one often obtains fractional dimensions for numerous exotic self-similar objects. We will investigate this phenomenon more closely in the next section.

To conclude this section, we collect some important properties of the Hausdorff dimension, which can be easily derived from Definition~\ref{def:Hdim} and the properties of Hausdorff measures listed in Remark~\ref{remark:Hausdorff-Lebesgue} and Lemma~\ref{lemma:hausdorff-scaling}.
\begin{theorem}
The Hausdorff dimension enjoys the following properties:
\begin{itemize}
    \item $\operatorname{dim}_H B = 0$ for any countable $B$.
    \item $\operatorname{dim}_H A \leq \operatorname{dim}_H B$ if $A\subseteq B$.
    \item $\operatorname{dim}_H (A\cup B) = \operatorname{max}\{\operatorname{dim}_H A, \operatorname{dim}_H B \}$.
    \item $\operatorname{dim}_H B \leq k$ for all subsets $B\subseteq \mathbb{R}^k$.
     \item $\operatorname{dim}_H f[B] \leq \operatorname{dim}_H B$ for Lipschitz $f:B \mapsto \mathbb{R}^l$, where $B\subseteq \mathbb{R}^k$.
    \item $\operatorname{dim}_H B = k$ for any subset $B\subseteq \mathbb{R}^k$ which has a positive $k$-dimensional Lebesgue measure. In particular, $\operatorname{dim}_H B = k$ for all open sets $B\subseteq \mathbb{R}^k$ and $\operatorname{dim}_H \mathbb{R}^k = k$.
\end{itemize}
\end{theorem}

\begin{remark} \label{remark:haus=topo}
The above theorem tells us that the topological dimension coincides with the Hausdorff dimension for the familiar geometric shapes like lines, planes, disks, cubes, and so on.
\end{remark}

\section{Examples} \label{sec:eg}

In this section, we lay out a general protocol to compute the Hausdorff dimension of a large class of self-similar objects. Before delving into the standard method, it is instructive to first look at the quintessential example of a self-similar set: the \emph{Cantor set}.

\begin{example}[\emph{Cantor set}] \label{eg-cantor}

Let us first recall the definition of the Cantor set. Beginning with the closed unit interval $C_0 = [0,1] \subset \mathbb{R}$, we first delete the open middle third interval $(1/3,2/3)$ from $C_0$ to obtain $C_1 = [0,1/3] \cup [2/3,1]$. Then, the open middle third interval of each of these two closed intervals is deleted to obtain $C_2$. Proceeding in this way, we define $$ C_n = \frac{C_{n-1}}{3} \cup (\frac{2}{3} + \frac{C_{n-1}}{3}) \qquad n\in \mathbb{N}.$$ The Cantor set $C$ is then defined as the infinite intersection $C = \bigcap_{n\in \mathbb{N}} C_n$. \\[0.2cm]
\begin{figure}[H]
    \centering
    \includegraphics[scale=0.8]{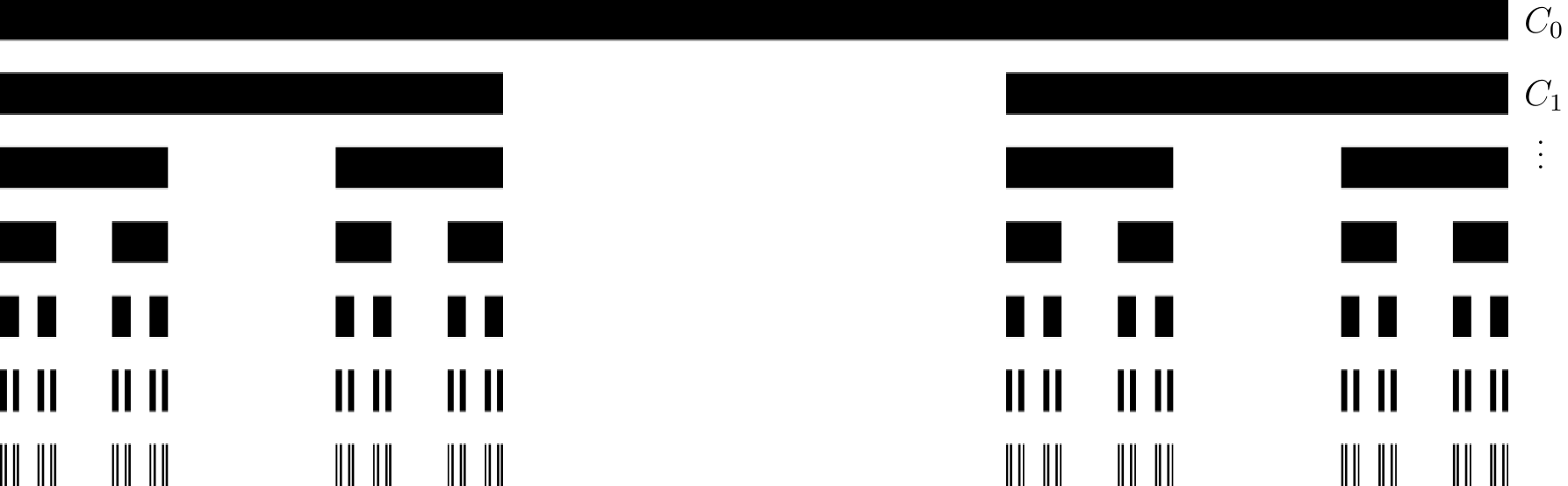}
    \caption{The first few steps in the construction of the Cantor set.}
    \label{fig:my_label}
\end{figure}
Being a countable intersection of closed (Borel) sets, the Cantor set itself is a Borel subset of $\mathbb{R}$ and thus lies in the domain of all $\delta$-Hausdorff measures. Moreover, it is clear that $C$ can be written as the disjoint union $C = C' \cup C''$, where $C' = C \cap [0,1/3] = C/3$ and $C'' = C' + 2/3$. Now, if we consider some $\delta$-Hausdorff measure on $\mathbb{R}$ for $\delta>0$, we can use the additivity (since $H^\delta$ is a measure) and the scaling property of $H^\delta$ (see Corollary~\ref{corollary:hausdorff-scaling}) to write
\begin{align*}
    H^\delta (C) = H^\delta (C') + H^\delta (C'') = H^\delta(C/3) + H^\delta (C/3 + 2/3) = \frac{2}{3^\delta} H^\delta (C)
\end{align*}
Assuming $H^\delta (C)$ is finite, we deduce that $2 = 3^\delta \implies \delta = \log 2 / \log 3$. Indeed, it can be checked that for $\delta = \log2 / \log 3$, $H^\delta (C)=1$. It is hence evident that $\operatorname{dim}_H C = \log 2 / \log 3$. The fractional Hausdorff dimension of the Cantor set tells us that the set is neither a countable collection of points (i.e. it is uncountably infinite) nor is it a continuous curve. It lies somewhere in between these two, hence the fractional dimension. It would be instructive for the reader to show that the topological dimension of $C$ is zero.
\end{example}

Taking inspiration from the above example, we now make the notion of a self-similar set in $\mathbb{R}^k$ more precise as follows. Consider a closed set $D\subseteq \mathbb{R}^k$ and a set of contractions (also called \emph{iterated function systems}) $\{S_i : D\rightarrow D \}_{i=1}^n$ such that $||S_i(x)-S_i(y)|| \leq r_i ||x-y|| \,\, \forall x,y\in D$, where $0<r_i<1$ are known as the \emph{scale factors} of the corresponding contractions $S_i$. Then, a compact set $F\subseteq D$ is said to be an \emph{attractor} if $F=\cup_{i=1}^n S_i(F)$. Since $F$ can be written as a union of shrunken copies of itself, it is clear that $F$ is self-similar. Notice that in Example~\ref{eg-cantor}, if $D$ is any closed set containing the unit interval $[0,1]$ and $S_1 (x) = x/3$, $S_2(x) = x/3 + 2/3 \,\,\, \forall x\in D$, the Cantor set $C$ becomes the unique attractor $C= S_1(C) \cup S_2 (C)$ of the system $\{S_1,S_2 \}$.

Now, the fundamental property of such a set of contractions is that it uniquely determines a non-empty compact set $F\subseteq D$ as its attractor, see the remarkable paper \cite{Hutch1981attractor} or \cite[Theorem 9.1]{falconer1990fractal}. Moreover, if the sets $S_i(F)$ are also pairwise disjoint, then we can apply the reasoning of Example~\ref{eg-cantor} to deduce that,
\begin{equation}
    H^\delta(F) = \sum_{i=1}^n r_i^\delta H^\delta (F).
\end{equation}
Notice that since all compact sets are Borel, $F$ lies in the domain of definition of $H^\delta$ for every $\delta>0$. Moreover, if we assume that $H^\delta(F)$ is finite, we can cancel it from both sides of the above equation to infer that the Hausdorff dimension of $F$ is the unique solution $d$ of the following equation: 
\begin{equation}
\sum_{i=1}^n r_i^d = 1.    
\end{equation}
Now, in real world scenarios, the attractors $F$ don't usually give rise to pairwise disjoint sets $S_i(F)$ and hence the above reasoning cannot be directly applied in such cases. However, one can relax this \emph{exact pairwise disjointness} condition by an \emph{approximate} one, which amounts to saying that for a given system of contractions $\{S_i : D\rightarrow D \}_{i=1}^n$, there exists a bounded, non-empty open set $V$ such that $\cup_{i=1}^n S_i(V) \subseteq V$, where the union is over disjoint sets (this is known as \emph{Moran's open set condition}, see \cite[Section 6.5]{edgar2008measure} or \cite[Section 9.2]{falconer1990fractal}). This also guarantees that the $\delta$-Hausdorff measure takes a finite value when $\delta=\operatorname{dim}_H F$. We formulate the above discussion more succinctly in the following theorem and refer the reader to \cite[Theorem 9.3]{falconer1990fractal} for a (rather technical) proof.

\begin{theorem} \label{theorem:self-similar}
Let $D\subseteq \mathbb{R}^k$ be a compact set. Consider a system of contractions $\{S_i : D\rightarrow D \}_{i=1}^n$ with the corresponding scaling factors $r_i\in (0,1)$ such that there exists a bounded no-empty open set $V$ with the property that $\cup_{i=1}^n S_i(V)$ is a disjoint union contained in $V$. Then, the Hausdorff dimension of the associated attractor $F$ is given by the unique solution $d$ of the equation $$\sum_{i=1}^n r_i^d = 1.$$ Moreover, $H^\delta (F)$ takes a finite value for $\delta=d$.
\end{theorem}

To conclude this article, we compute the Hausdorff dimension of several self-similar structures to show how Theorem~\ref{theorem:self-similar} is employed in practice.

\begin{example}[Koch curve]
Start with the unit interval $K_0 = [0,1]$. divide it into 3 equal parts and replace the middle part with an equilateral triangle with its base fixed in the position of the removed interval. Remove the base of the added triangle again to obtain the curve $K_1$ as a union of four equal intervals. Repeat this step for each of these intervals to obtain $K_2$, and iterate this process indefinitely to construct the \emph{Kock curve}, see Figure~\ref{fig:koch-bnw}. 

\begin{figure}[H]
    \centering
    \includegraphics[scale=0.6]{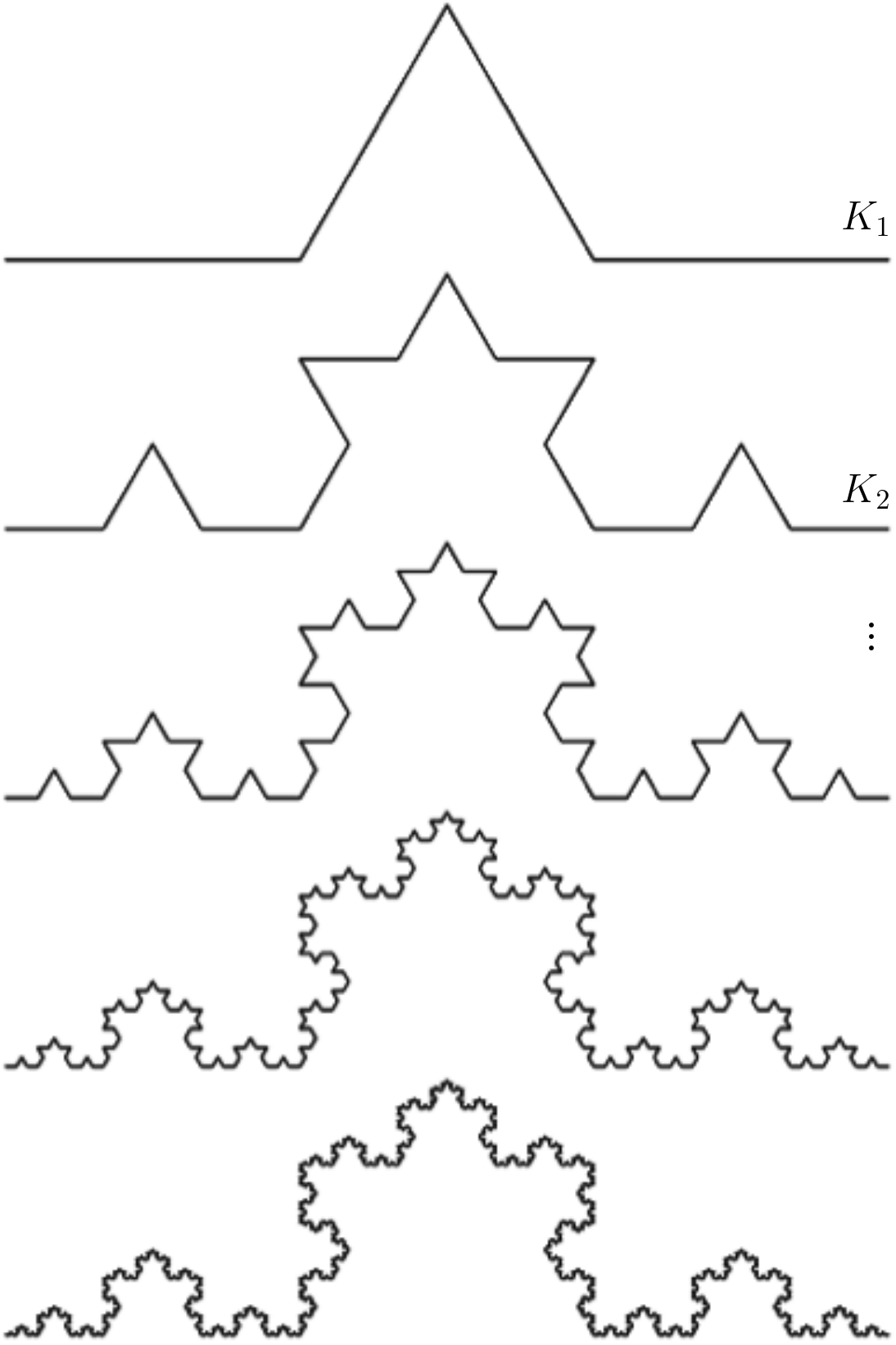}
    \caption{The first few iterations in the evolution of the Koch curve}
    \label{fig:koch-bnw}
\end{figure}

Alternatively, one can define $K$ as the attractor of the contractions mapping the unit interval $[0,1]$ onto each of the four intervals in $K_1$, see Figure~\ref{fig:koch-color}. Taking $V$ to be the (topological) interior of the isosceles triangle with base $[0,1]$ and height $1/2\sqrt{3}$, it is clear that these contractions satisfy Moran's open set condition. Finally, since the scaling factor $r=1/3$ is the same for all 4 contractions, we deduce that the Kock curve has Hausdorff dimension $\operatorname{dim}_H K = \log 4/\log 3$, see Theorem~\ref{theorem:self-similar}.

\begin{figure}[H]
    \centering
    \includegraphics[scale=0.7]{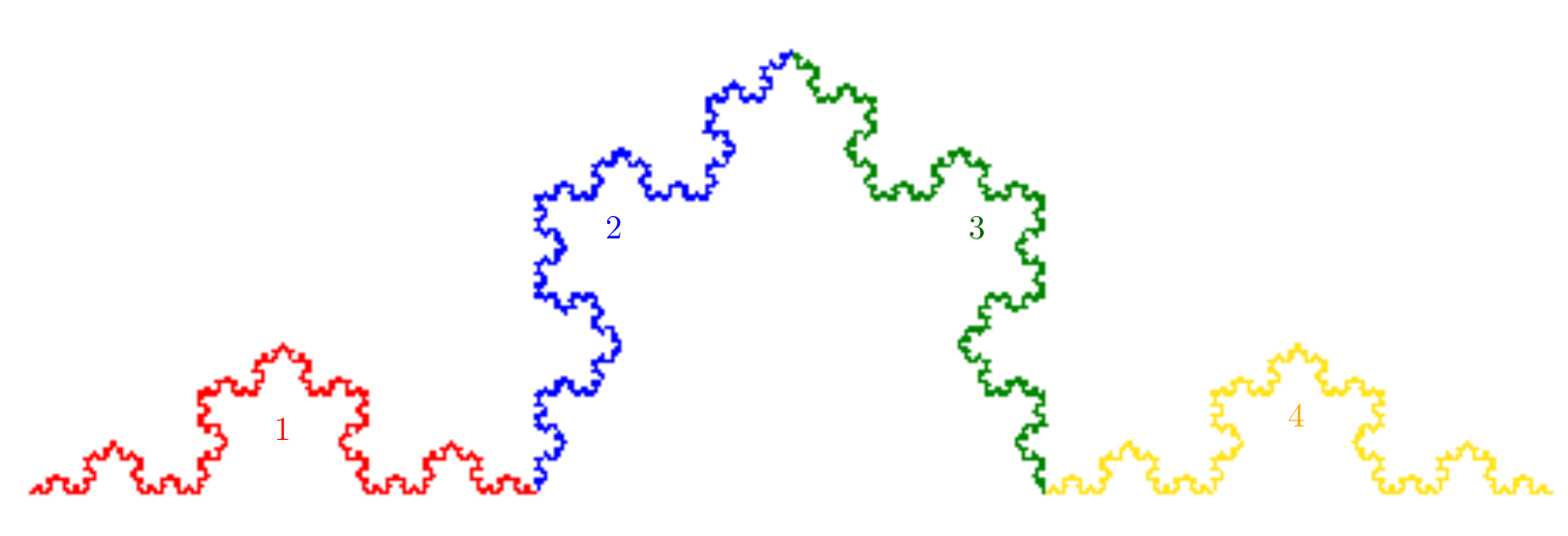}
    \caption{The Koch curve $K$, expressed as the attractor of the set of 4 contractions $\{S_i\}_{i=1}^4$, each of which maps the unit interval onto one of the 4 intervals in $K_1$. Different colored parts of the curve denote the different images $S_i(K)$ of the Koch curve under these contraction, so that $K=\cup_{i=1}^4 S_i(K)$}.
    \label{fig:koch-color}
\end{figure}
\end{example}

\begin{example}[Sierpiński triangle]
Start with an equilateral triangle $T_0$ with base equal to the unit interval $[0,1]$. Shrink it to an equilateral triangle with half the base length. Make 3 copies of this shrunken triangle and arrange them so that each triangle touches the other two at the corners. (Notice the emergence of the inverted triangle hole in the middle). Call this arrangement $T_1$. Repeat the above step for each of the 3 triangles in $T_1$ to obtain $T_2$ and iterate this process indefinitely to obtain the \emph{Sierpiński triangle} $T$, see Figure~\ref{fig:Striangle}. 
\begin{figure}[H]
    \centering
    \includegraphics[scale=0.8]{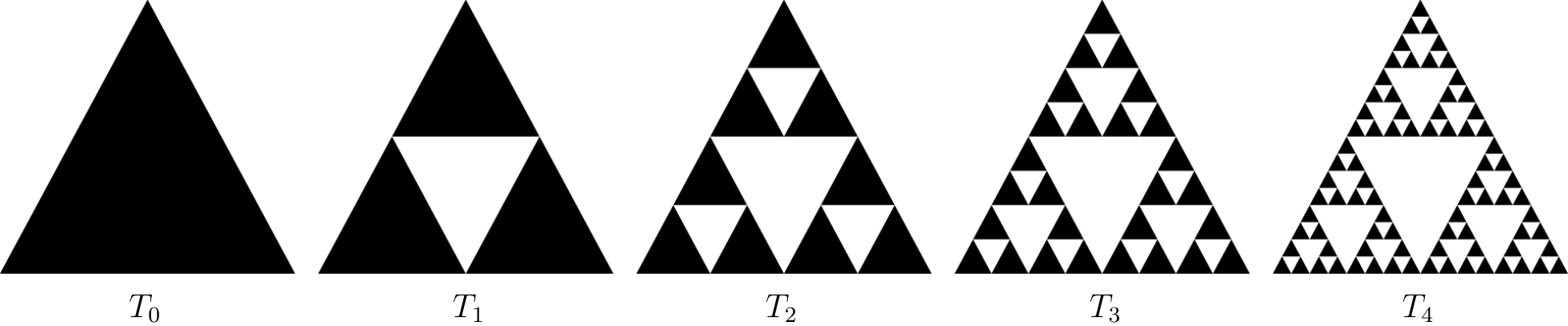}
    \caption{The first few iterations in the evolution of the Sierpiński triangle.}
    \label{fig:Striangle}
\end{figure}
It is clear that $T$ can also be defined as the attractor of the set of contractions with scaling factor $r=1/2$ that map the $T_0$ triangle onto each of the 3 smaller ones in $T_1$. Moreover, Moran's open set condition holds, as can be seen by taking $V$ to be the (topological) interior of $T_0$ triangle. Hence, we can immediately apply Theorem~\ref{theorem:self-similar} again to obtain the Hausdorff dimension $\operatorname{dim}_H T = \log 3/\log 2$.
\begin{figure}[H]
    \centering
    \includegraphics[scale=0.75]{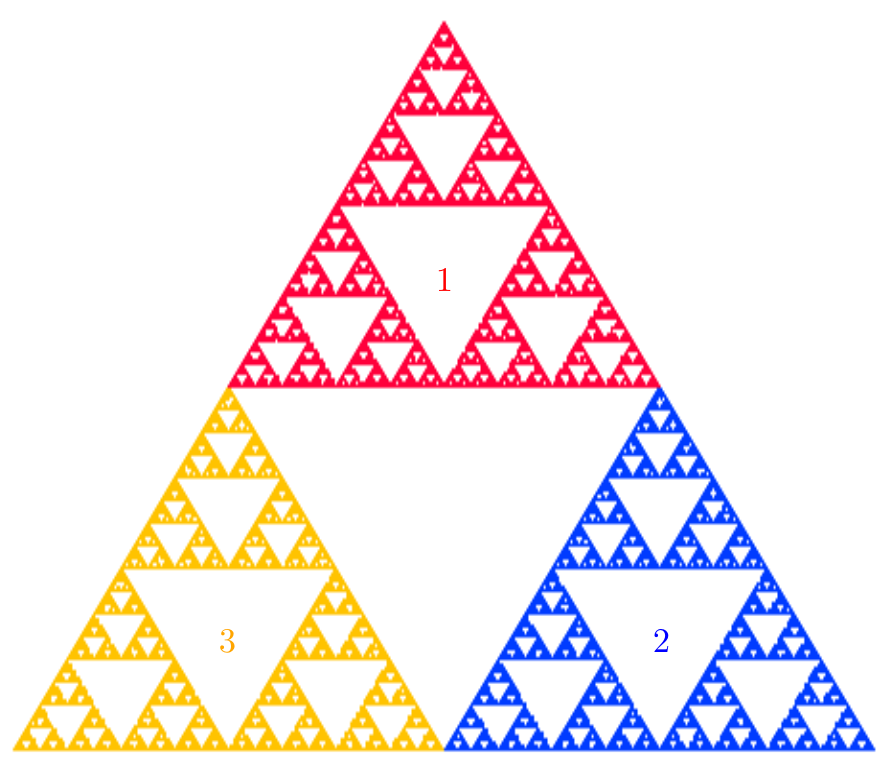}
    \caption{The Sierpiński triangle $T$, expressed as the attractor of the set of 3 contractions $\{S_i\}_{i=1}^3$, each of which maps the $T_0$ triangle onto one of the 3 triangles in $T_1$. Different colored parts of the triangle denote the different images $S_i(K)$ of the Sierpiński triangle under these contraction, so that $T=\cup_{i=1}^3 S_i(T)$.}
    \label{fig:my_label}
\end{figure}
\end{example}

\begin{remark}
It would be edifying for the reader to compute the topological dimensions of the objects that have been considered in this section and check that these are strictly less than the corresponding Hausdorff dimensions. In fact, this property is a fundamental feature of \emph{fractals}, a class of objects with an inherent local roughness and structure built into them which stays undetected through the topological lens of dimension. While some perfectly self-similar sets -- like the ones considered in this section -- can be classified as fractals, this is not necessarily true in general (consider a line or a plane for instance, which are certainly self-similar but have the same topological and Hausdorff dimensions). Unquestionably, there's more to the theory of fractals than what has been said here, and we refer the interested readers to Mandelbrot's marvellous book \cite{mandelbrot1983fractal} for further details.

\end{remark}

\bibliographystyle{alpha}
\bibliography{references}

\hrule
\bigskip
\end{document}